\documentclass{amsart}
\usepackage{amsmath, amsthm, amssymb}
\usepackage{verbatim}
\usepackage{tikz}
\usetikzlibrary{matrix}

\usepackage[mathscr]{eucal} 
\usepackage{mathrsfs} 
\usepackage{cancel}

\newcommand{\veps}{\varepsilon}
\newcommand{\vphi}{\varphi}

\newcommand{\al}{\alpha}

\newcommand{\la}{\lambda}
\newcommand{\om}{\omega}

\newcommand{\cP}{\mathcal{P}}
\newcommand{\cF}{\mathcal{F}}

\newtheorem{thm}{Theorem}
\newtheorem{prop}[thm]{Proposition}
\newtheorem{lem}[thm]{Lemma}
\newtheorem{cor}[thm]{Corollary}

\theoremstyle{definition}

\newtheorem{remark}[thm]{Remark}

\numberwithin{thm}{section}
\numberwithin{equation}{section}
\renewcommand{\[}{\begin{equation}}
\renewcommand{\]}{\end{equation}}

\newcommand{\wed}{\wedge}

\title[Weak convergence of Monge-Amp\`ere measures]{Weak convergence of Monge-Amp\`ere measures on compact Hermitian manifolds}

\author{S\l awomir Ko\l odziej and Ngoc Cuong Nguyen} 
\address{Faculty of Mathematics and Computer Science, Jagiellonian University, \L ojasiewicza 6, 30-348 Krak\'ow, Poland}
\email{slawomir.kolodziej@im.uj.edu.pl}
\address{Department of Mathematical Sciences, KAIST, 291 Daehak-ro, Yuseong-gu, Daejeon 34141, South Korea}
\email{cuongnn@kaist.ac.kr}
\begin{document}

\maketitle

\begin{abstract}
We give a sufficient condition on a sequence of uniformly bounded $\om$-plurisubharmonic functions, $\om$ being a Hermitian metric, for which the sequence of associated Monge-Amp\`ere measures converges weakly. This criterion can be used to obtained a bounded $\om$-plurisubharmonic solution to the Monge-Amp\`ere equation.
\end{abstract}

\bigskip
\bigskip

\section{Introduction}

Let $(X, \om)$ be a compact Hermitian manifold of dimension $n$. The complex Monge-Amp\`ere equation  on those manifolds has been studied extensively in recent years. The classical solution in the smooth case was provided by Tosatti-Weinkowe \cite{TW10b} (for $n=2$ the equation was solved earlier by Cherrier \cite{Ch87}). The weak solutions to the equation were studied in \cite{DK12}, \cite{KN1, KN4, KN2, KN7}, \cite{LPT20} and recent advances for the semi-positive Hermitian form are contained in  \cite{GL21b}. The equation has found numerous geometric applications \cite{To15}, \cite{Ng16}, \cite{Di16}, \cite{Ni17}, \cite{T18} \cite{KT21} and \cite{GL21a}.

To obtain the weak solutions one often uses the stability of potentials of the approximating equations. 
It is well-known that the $L^p$-convergence of potentials does not imply the weak convergence of corresponding Monge-Amp\`ere measures. Furthermore, in the Hermitian setting there are fewer such criterions available as  compared to the K\"ahler manifolds.  In this note, we prove such a criterion  under suitable assumptions: of uniform boundedness and domination by Monge-Amp\`ere measures of another "nice" sequence. Recall  the Bedford-Taylor capacity: for a Borel subset $E\subset X$,
$$	cap_\om(E) := \sup\left\{\int_E (\om + dd^cw)^n : w\in PSH(X, \om), 0\leq w\leq 1\right\}.
$$
Here $PSH(X,\om)$ denotes the set of all $\om$-plurisubharmonic ($\om$-psh) functions on $X$. 
A sequence $\{\vphi_j\}_{j=1}^\infty \subset PSH(X,\om)$ is said converge in capacity to $\vphi\in PSH(X,\om)$ if for a given $\veps>0$, 
$$	
	\lim_{j\to +\infty} cap_\om(|\vphi_j -\vphi|>\veps) =0.
$$
The main result of this expository note is as follows.

\begin{thm} \label{thm:main} Let $\{u_j\}$ be a uniformly bounded sequence of $\om$-psh functions. Assume $(\om +dd^c u_j)^n \leq C (\om + dd^c \vphi_j)^n$  for some uniformly bounded sequence $\{\vphi_j\}$ such  that $\vphi_j \to \vphi \in PSH(X,\om)$ in capacity. If $u_j \to u \in PSH(X,\om) \cap L^\infty(X)$ in $L^1(X)$, then  a subsequence of $(\om+ dd^c u_j)^n$ converges weakly  to $(\om +dd^cu)^n$. 
\end{thm}
This is a generalization of \cite[Lemma~2.1]{CK06} from the local setting to compact Hermitian manifolds that is pointed out in the proof of  \cite[Lemma~2.11]{KN22} (see also \cite{KN21}). On compact K\"aher manifolds there are stronger results \cite[Theorem~2.1]{Hiep08} and \cite{DH12}. We will see that if the dimension $n=2$, then we have a similar result (Proposition~\ref{prop:cap-convergence}). However, we do not know if it still holds for dimensions $n\geq 3$.
It is also possible to extend the  theorem to  the case of Hermitian semi-positive $(1,1)$-forms $\al$ and we will consider it in a future  paper. 

An application of the theorem is that it provides a shorter proof of the existence of bounded $\om$-psh solutions to Monge-Amp\`ere equations with the right hand side in $L^p(X)$,  $p>1$.

\begin{cor} Let $0\leq f \in L^p(X)$ and $p>1$. Then, there exists a bounded function $u \in PSH(X,\om) \cap L^\infty(X)$ and a constant $c>0$ solving $\om_u^n = c f \om^n.$
\end{cor}

\begin{proof} Approximating $f_j \to f$ by smooth positive function $f_j$. By the Tosatti-Weinkove theorem \cite{TW10b} there are  $u_j \in PSH(X,\om) \cap C^\infty(X)$ and $c_j>0$ solving
$$
	(\om + dd^c u_j)^n = c_j f_j \om^n, \quad  \sup_X u_j =0.
$$
Using \cite[Eq. (5.12)]{KN1} we have $c_j \leq C_0$ with a uniform $C_0>0$. Therefore, by \cite{DK12} we have  $- C_1 \leq u_j \leq 0$.  By passing to a subsequence we may assume that $u_j \to u$ in $L^1(X)$ and $c_j \to c \geq 0$. Then, 
$$
	\int_{X} |u_j-u| (\om+ dd^c u_j)^n = \int_X |u-u_j| c_j f_j \om^n  \leq  C \|u_j -u\|_{L^1(X)}^\frac{1}{q},
$$
where $1/p+1/q =1$.
Hence, by Lemma~\ref{lem:weak-convergence} (below) there is a subsequence of $\{u_j\}$, for simplicity still denoted by $\{u_j\}$, such that $\om_{u_j}^n \to \om_u^n = c f \om^n$ weakly. Note that since $u$ is bounded,  it follows from \cite[Remark~5.7]{KN1} that $c>0$.
\end{proof}

Another application is proving the existence of bounded solutions for  a more general class of measures.
A positive Radon measure $\mu$ on $X$ is well dominated by capacity or belong to $\cF(X,h)$ (see \cite{K05})  if
\[
\label{eq:dominate}
	\mu(E) \leq F_h( cap_\om (E)),
\]
for any Borel set $E \subset X$, where $F_h(x) = x/h(x^{-\frac{1}{n}})$ for some increasing 
function $h : \mathbb R_+ \rightarrow (0, \infty ) $ satisfying
\[
\label{eq:admissible}
	\int_1^\infty \frac{1}{x [h(x) ]^{\frac{1}{n}} }  \, dx < +\infty.
\]

\begin{cor} Let $\mu$ be a positive Radon measure with $\mu(X) >0$. Assume that $\mu$ is well dominated by capacity. Then, there exists a bounded solution $u\in PSH(X,
\om)$ and $c>0$ solving 
$
	\om_u^n = c \mu.
$
\end{cor}

\begin{proof} It follows easily from the proof of \cite[Theorem~3.1]{KN21}. % that there exists an approximate sequence $\mu_\veps \to \mu$ weakly as $\veps \searrow 0$.
\end{proof}
\begin{remark} The solutions obtained in Corollary~2.7 and Corollary~2.8 are H\"older continuous and continuous, respectively. However, to get these statements, we need a stronger stability of solutions as in \cite{KN1, KN7}.
\end{remark}

\medskip

{\em Acknowledgement.} The first author is partially supported by  grant  no. \linebreak 2021/41/B/ST1/01632 from the National Science Center, Poland. The second author is  partially supported by the start-up grant G04190056 of KAIST and the National Research Foundation of Korea (NRF) grant  no. 2021R1F1A1048185.

\section{Proof of Theorem~\ref{thm:main} and convergence in capacity}

This section is devoted to  prove the main theorem.
Since all functions under consideration are uniformly bounded, by normalizing 
\[ \sup_X u_j = \sup_X \vphi_j =0,
\]
they belong to the  subset $\cP_0$ of $PSH(X,\om)$ given by
$$\cP_0 =\left\{  v\in PSH (X,\om) \cap L^\infty(X):  \sup_X v=0 \right\}.$$
Let us denote  $dV = \om^n$ the volume form of $X$.
We have the following result essentially due to Cegrell \cite{Ce98},  a detailed proof is  given in \cite[Lemma~2.1]{KN21}. 
\begin{lem} \label{lem:L1-norm-convergence}
Let $d\la$ be a finite positive Radon measure on $X$ vanishing on pluripolar sets. Suppose moreover that $\{u_j\} \subset  \cP_0$  converges $dV$-a.e. to  $u \in \cP_0$. Then there exists a subsequence  $\{u_{j_s}\} \subset \{u_j\}$ such that
$$  \lim_{j_s\to +\infty} \int_X u_{j_s} d\la = \int_X u d\la.$$
\end{lem}

\begin{comment}%%%%%%%%%%%%%%%%%%%%%%%%%%%%
\begin{proof} Since $d\la$ is a finite measure it follows that $\sup_{j} \int_{X} |u_j|^2 d\la < +\infty$. So there exists a subsequence $\{u_j\}$ weakly converging to $v \in L^2 (d\la)$. By the Banach-Saks theorem we can find a subsequence $u_{j_k}$ such  that 
$$	
	F_k= \frac{1}{k} (u_{j_1} + \cdots + u_{j_k}) \to v \quad\text{in } L^2(d\la)
$$
as $k \to +\infty$. Extracting a subsequence $\{F_{k_s}\}$ of $\{F_k\}$ we  get $F_{k_s} \to v$ a.e in $d\la$, and also  that $F_{k_s}$ converges a.e to $u$ with respect to the Lebesgue measure.  Therefore, $(\sup_{s>t} F_{k_s})^* \searrow u$ everywhere as $t\to +\infty$.
It follows that  there is a subsequence which we still denote by $\{u_j\}$ such that
$$
	\lim_{j\to \infty} \int_X u_j d\la = \int_X v d\la = \lim_{s\to \infty} \int_X F_{k_s} d\la = \lim_{t\to \infty} \int_X \sup_{s>t} F_{k_s}  d\la = \int_X u d\la.
$$
where the first identity used the decreasing convergence property; the second one used the $d\la$-a.e convergence, and  the last used the fact that $d\la$ does not charge  pluripolar sets. This completes the proof.
\end{proof}
\end{comment}%%%%%%%%%%%%%%%%%%%%%%%%%%%%%%

An immediate consequence is

\begin{cor}\label{cor:L1-convergence}  There exists a subsequence, still denoted by $\{u_j\}$, such that
$$ \lim_{j\to \infty} \int_X |u_j - u| d\la  =0.
$$
\end{cor}

\begin{proof}Applying Lemma~\ref{lem:L1-norm-convergence} twice to the sequences $\{u_j\}$ and $\max\{u_j, u\}$,  we have (still denoting by $\{u_j\}$  the resulting subsequence)
$$\lim_{j\to \infty} \int_X u_j d\la = \int_X u d\la, \quad \lim_{j\to \infty} \int_X \max\{u_j, u\} d\la = \int_X u d\la.$$
Since $\max\{u,u_j\} = (u_j+u + |u_j-u|)/2$, integrating both sides and using the previous equations we get that
 $u_j \to u$ in $L^1(d\la).$
\end{proof}

The next result  is a global analogue of \cite[Lemma~2.3]{KN22} which says that under the assumption of Theorem~\ref{thm:main}, the sequence converges uniformly in $L^1$-norm with respect to a family of measures.  On compact complex manifolds the Cegrell inequality will not be needed. For the reader convenience, we give details of the proof.

\begin{lem} \label{lem:uniform-L1-convergence} Let $\{w_j\}_{j=1}^\infty \subset \cP_0$ be a uniformly bounded sequence  that converges in capacity to $w\in \cP_0$. Assume   $\sup_j \int_X (\om+ dd^c w_j)^n \leq C_1$ for some $C_1>0.$ Then, 
$$
	\lim_{j\to \infty}  \int_X |u- u_j| (\om+ dd^c w_j)^k \wed \om^{n-k} = 0 \quad \text{for } k=0,1,...,n.
$$
\end{lem}

\begin{proof} The case $k=0$ is the assumption on $\{u_j\}$, and we only give here the proof of the last inductive step, the other steps are  the very similar. In the proof  that follows the constant $C>0$ is a  uniform constant depending only on $X, \om$ and the uniform bounds for $\|u_j\|_{L^\infty}$ and $\|w_j\|_{L^\infty}$, it may change from line to line.

Note that $|u-u_j| = (\max\{u, u_j\} - u_j) + (\max\{u, u_j\} - u)$. By quasi-continuity of $\om$-psh functions and Hartogs' lemma,  we have $\phi_j := \max\{u, u_j\} \to u$ in capacity. Fix $\veps>0$. Then,  when $j $ is large,
$$\begin{aligned}
	\int _{X} (\max\{u, u_j\} -u) (\om+dd^c w_j)^n 
&\leq \int_{\{|\phi_j -u|>\veps\}} (\om+dd^c w_j)^n + \veps \int_{X} (\om+dd^c w_j)^n \\
&\leq C cap_\om (|\phi_j -u|>\veps) + C_1 \veps.
\end{aligned}$$
Therefore, $\lim_{j\to \infty}  \int_X (\phi_j-u) (\om+dd^c w_j)^n =0$. Next, we have for $j>k$,
$$
\int_X (\phi_j -u_j) (\om+dd^c w_j)^n - \int_X (\phi_j - u_j) (\om+dd^c w_k)^n 
= \int_X (\phi_j - u_j) dd^c (w_j - w_k) \wed T ,
$$
where $T= T(j,k) = \sum_{s=1}^{n-1} \om_{w_j}^s \wed \om_{w_k}^{n-1-s}$. By integration by parts, 
$$\begin{aligned}
	\int_X (\phi_j - u_j) dd^c (w_j - w_k) \wed T 
&= \int_X (w_j - w_k) dd^c \left[(\phi_j -u_j) \wed T \right].
\end{aligned}
$$
Note that for $h= \phi_j -u_j$,
$$ \begin{aligned}
	dd^c (h T) 
&= 	dd^c h \wed T + dh \wed d^c T - d^c h \wed dT + h dd^c T  \\
&=	dd^c h \wed T + dh \wed d^c \om \wed T_1 - d^c h \wed d\om \wed T_2 \\
&=: S_0 + S_1 + S_2.
\end{aligned}$$
Here notice that $T_1$ and $T_2$ are positive currents, and similar type as of $T$.

We now estimate each term $S_0, S_1$ and $S_2$ separately. Firstly, for the term $S_0$, 
$$
	\int_X (w_j - w_k) dd^c (\phi_j-u_j) \wed T \leq \int_X |w_j -w_k| (\om_{\phi_j}+ \om_{u_j}) \wed T.
$$
Since $\|w_j \|_\infty, \|u_j\|_\infty \leq A$ in $X$, it follows that 
$$\begin{aligned}
	 \int_{X} |w_j - w_k| (\om_{\phi_j} + \om_{u_j}) \wed T 
&\leq  A \int_{\{|w_j -w_k| >  \veps\}}  (\om_{\phi_j} + \om_{u_j}) \wed T  \\ &\quad+ \veps\int_{\{|w_j -w_k| \leq \veps\}}  (\om_{\phi_j} + \om_{u_j}) \wed T  \\
&\leq  A^{n+1} cap_\om( |w_j -w_k| >  \veps)  + 	C \veps,
\end{aligned}$$
where the uniform bound for the second integral on the right hand side follows from uniform boundedness of the potentials (see. e.g., \cite[Proposition~2.3]{DK12}).
Since $w_j\to w$ in capacity, it follows that there exists $k_0$ such that for every $j>k\geq k_0$  the left hand side is less than $2C \veps$.

Secondly, for the term $S_1$, by the Cauchy-Schwarz inequality \cite[Proposition~1.4]{Ng16} in the Hermitian setting,
$$ \begin{aligned}
&	\left|\int_X (w_j - w_k) dh \wed d^c \om \wed T_1 \right|^2 \\
&\leq 	C \int_X |w_j-w_k| dh\wed d^c h \wed \om \wed T_1 \int_X |w_j-w_k| \om^2 \wed T_1 \\
&\leq 	C \int_X |w_j-w_k| \om^2 \wed T_1.
\end{aligned}$$
Therefore, by a similar argument as in the first case for the integral on the right hand side is bounded by $C\veps$ for every $j> k \geq k_0$ (we may increase $k_0$ if necessary).
 
Lastly, the term $S_2$ is estimated similarly as  $S_1$.
Thus, 
$$\begin{aligned}
\int_X (\phi_j -u_j) (\om+ dd^c w_j)^n 
&\leq  \int_X (\phi_j -u_j) (\om+ dd^c w_k)^n \\ 
&\quad +	\left|\int_X (\phi_j -u_j) \om_{w_j}^n - \int_X (\phi_j - u_j) \om_{w_k}^n \right| \\
&\leq  \int_X (\phi_j -u_j) (\om+ dd^c w_k)^n  + 2C \veps \\
&\leq \int_X |u-u_j| (\om+ dd^c w_{k})^n + 2 C \veps.
\end{aligned}$$
Fixing $k=k_0$ and applying Corollary~\ref{cor:L1-convergence} for $d\la = (\om+dd^c w_{k_0})^n$, we get that 
$$
	 \int_X (\phi_j -u_j) (dd^c w_j)^n  \leq (2C + 1) \veps \quad\text{for } j \geq k_1 \geq k_0.
$$
Since $\veps>0$ is arbitrary, the proof of the lemma is completed.
\end{proof}

The proof of Theorem~\ref{thm:main} is an immediate consequence of Lemma~\ref{lem:uniform-L1-convergence} and the following 
 weak convergence. 

\begin{lem} \label{lem:weak-convergence}
Suppose that 
\[\label{eq:energy-convergence}
	\lim_{j\to +\infty} \int_X |u_j -u| \om_{u_j}^n =0.
\]
Then, there exists a subsequence $\{u_{j_s}\}$ of $\{u_j\}$ such that $\om_{u_{j_s}}^n$ converges to $\om_u^n$ weakly. 
\end{lem}

\begin{proof} 
Let $A>0$ be such that  $-A \leq u_j, u \leq 0$.
By passing to a subsequence we may assume further that $u_j \to u$ a.e. in $X$ with respect to $dV$. 
Note that
 $$u = (\limsup_{j\to \infty} u_j)^* = \lim_{j\to \infty} (\sup_{\ell \geq j} u_\ell)^*.$$
Set 
\[\label{eq:hartogs-s}
	w_{j} =\max \{u_j, u-1/j\}.
\]
By the Hartogs lemma $w_j$ converges to $u$ in capacity. Therefore, by the convergence theorem in \cite{DK12} (see also \cite{BT82}),  
$
	\lim_{j\to \infty} \om_{w_j}^n = \om_u^n.
$
Thanks to Lemma~\ref{lem:uniform-L1-convergence} and the assumption $w_j \to u$ in capacity, we have
\[\label{eq:uniform-L1}
	\int_{X} |u_j -u| (\om + dd^c w_j)^n \to 0 \quad \text{as } j\to +\infty.
\]

Now we are ready to conclude the proof of the lemma. By \eqref{eq:energy-convergence} and \eqref{eq:uniform-L1} we can choose  a subsequence $\{u_{j_s}\} \subset \{u_j\}$ so that $$\int_X |u-u_{j_s}| (\om + dd^c u_{j_s})^n + \int_X |u-u_{j_s}|(\om + dd^c w_s)^n < 1/s^2.$$ 
Recall from \eqref{eq:hartogs-s} that $w_s = \max\{u_{j_s}, u-1/s\}$ and this implies
$$
	{\bf 1}_{\{u_{j_s} > u-1/s\}} (\om+ dd^c w_s)^n = {\bf 1}_{\{u_{j_s} > u-1/s\}} (\om + dd^c u_{j_s})^n.
$$ 
Therefore,  for $\eta \in C^\infty(X)$,
$$\begin{aligned}
	\left| \int_X \eta \om_{u}^n - \int_X \eta \om_{u_{j_s}}^n\right| 
&\leq  \left| \int_X \eta \om_{u}^n - \int_X \eta \om_{w_s}^n\right| + \left| \int_X \eta \om_{w_s}^n - \int_X \eta \om_{u_{j_s}}^n\right| \\
&\leq \left| \int_X \eta \om_{u}^n - \int_X \eta \om_{w_s}^n\right| + \left| \int_{\{u_{j_s} \leq u-1/s \}} \eta \om_{w_s}^n - \eta \om_{u_{j_s}}^n\right|.
\end{aligned}$$
The first term on the right hand side goes to zero as $\om_{w_s}^n \to \om_u^n$. It remains to estimate the second term. Firstly, by the choice of $\{u_{j_s}\} $ at the beginning of this proof,
$$\begin{aligned}
	\left|\int_{\{u_{j_s} \leq u-1/s \}} \eta \om_{u_{j_s}}^n \right| 
&\leq \|\eta\|_{L^\infty} \int_{\{u_{j_s} \leq u-1/s \}} \om_{u_{j_s}}^n \\
&\leq s \|\eta\|_{L^\infty} \int_X |u- u_{j_s}| \om_{u_{j_s}}^n  \leq \frac{1}{s}\|\eta\|_{L^\infty} \to 0 \quad\text{as } s\to +\infty.
\end{aligned}$$
Similarly, 
$$ \begin{aligned}
 \left|\int_{\{u_{j_s} \leq u-1/s\}} \eta \om_{w_s}^n \right| 
 &\leq \|\eta\|_{L^\infty} \left| \int_{\{u_{j_s}\leq  u-1/s\}} \eta \om_{w_s}^n\right|  \\
 &\leq s \|\eta\|_{L^\infty} \int_X |u- u_{j_s}| \om_{w_s}^n  \to 0 \quad\text{as } s\to +\infty.
\end{aligned}
$$
The last two estimates complete the proof of the lemma.  This also completed the proof of Theorem~\ref{thm:main}.
\end{proof}

We shall have a similar criterion for convergence in capacity \cite[Proposition~2.8]{KN21}. Let us recall \cite[Theorem~3.5]{DK12} and \cite[Lemmas~2.1, 2.2]{KN1}. 
Let $B>0$ be a uniform constant satisfying 
$$\begin{aligned}
&	-B\om^2 \leq 2n dd^c \om \leq B \om^2, \\
&	-B \om^3 \leq 4n^2  d\om \wed d^c \om \leq B\om^3.
\end{aligned}
$$
Then, for $\vphi, \psi \in PSH(X, \om) \cap L^\infty(X)$ satisfying $\sup_{\{\vphi<\psi\}} (\psi -\vphi) \leq 1$, 
\[ \label{eq:comparison-V1} \begin{aligned}
	\int_{\{\vphi<\psi\}} \om_\psi^n  &\leq \int_{\{\vphi < \psi\}} \om_{\vphi}^n 
	+ C_n \max\{1,B\}^n \sum_{k=0}^{n} \int_{\{\psi< \vphi\}} \om_{\vphi}^k \wed \om^{n-k},
\end{aligned}
\] 
where $C_n$ is a dimensional constant. If  the dimension $n=2$, then we have a better estimate, namely,
\[ \label{eq:comparison-V2} \begin{aligned}
	\int_{\{\vphi<\psi\}} \om_\psi^2  &\leq \int_{\{\vphi < \psi\}} \om_{\vphi}^2 
	+ C_n \max\{1,B\} \int_{\{\psi< \vphi\}} \om^2,
\end{aligned}
\]

\begin{prop} \label{prop:cap-convergence} Let $\{u_j\}$ be the sequence in Theorem~\ref{thm:main}. Then, $u_j$ converges to $u$ in capacity if and only if 
\[\label{eq:cap-convergence}
	\lim_{j\to +\infty} \int_X |u_j-u| \om_{u_j}^k \wed \om^{n-k} = 0, \quad k=0,...,n.
\]
In particular, if $n=2$, then the sequence converges in capacity.
\end{prop}

\begin{proof} Suppose that $u_j \to u$ in capacity. It follows from Lemma~\ref{lem:uniform-L1-convergence} that \eqref{eq:cap-convergence} holds for every $k=0, 1, ...,n$. Conversely, suppose that \eqref{eq:cap-convergence} holds true. Let  $\veps>0$ be fixed and without loss of generality we may assume that 
$
	-1 \leq u_j, u \leq 0.
$
We have
$$
	cap_\om(|u_j-u|>\veps) \leq cap_\om (u_j-u>\veps) + cap_\om (u-u_j >\veps).
$$
By Hartogs' lemma, $\max\{u_j, u\} \to u$ in capacity. Hence, $$\lim_{j\to +\infty} cap_\om (u_j-u >\veps) \leq \lim_{j\to +\infty} cap_\om(\max\{u_j,u\} -u >0) =0.$$
It remains to prove that $cap_\om(u-u_j>\veps) \to 0$ as $j\to +\infty$. Let $-1\leq \rho \leq 0$ be a function in $PSH(X,\om)$. Then,
$$
	\{u_j < u -\veps\} \subset \left\{u_j < (1-\frac{\veps}{4}) u + \frac{\veps}{4} \rho - \frac{3\veps}{4} \right\} \subset  \left \{ u_j < u - \frac{\veps}{2}\right\}
$$
Applying  \eqref{eq:comparison-V1} for $\vphi = u_j $ and $\psi = (1-\veps/4) u + \veps \rho/4 - 3\veps/4$ and then using the previous inclusions for domains of integration on the left hand side and on the right hand side, we have
$$\begin{aligned}
&	\int_{\{u_j < u-\veps\}} \left[\om 
+ (1-\frac{\veps}{4}) dd^c u + \frac{\veps}{4} dd^c  \rho \right]^n  \\
&\leq \int_{\{u_j< u - \veps/2\}} \om_{u_j}^n 
	 + C_n \max\{1,B\}^n \sum_{k=0}^{n} \int_{\{u_j< u-\veps/2\}} \om_{u_j}^k \wed \om^{n-k},
\end{aligned}$$
Since the integrand on the left hand side is larger than $(\veps/4)^n\om_\rho^n$, it follows that 
\[ \label{eq:cap-inequality} \begin{aligned}
	\left(\frac{\veps}{4} \right)^n cap_\om (u_j< u-\veps) &\leq \int_{\{u_j < u -\veps/2\}} \om_{u_j}^n \\ 
&\quad	+ C_n \max\{1,B\}^n \sum_{k=0}^{n} \int_{\{u_j< u-\veps/2\}} \om_{u_j}^k \wed \om^{n-k}.
\end{aligned}
\] 
Notice that for $k=0,...,n$ we have
$$
	\int_{\{u_j < u -\veps/2\}} \om_{u_j}^k \wed \om^{n-k}  \leq \frac{2}{\veps}  \int_X |u_j -u| \om_{u_j}^k \wed \om^{n-k} .
$$
Combining this inequality with \eqref{eq:cap-convergence} and \eqref{eq:cap-inequality} we obtain
$\lim_{j\to +\infty} cap_\om(u_j < u -\veps) =0$. The proof is completed.

Finally let us  finish the proof in the case $n=2$. Indeed, by \eqref{eq:comparison-V2} the  inequality corresponding to \eqref{eq:cap-inequality} reads
$$
	\left(\frac{\veps}{4} \right)^2 cap_\om(\{u_j < u-\veps\}) \leq \int_{\{u_j< u-\veps/2\}}\om_{u_j}^2 + C_n \max\{1,B\} \int_{\{u_j < u-\veps/2\}} \om^2.
$$
The right hand side is bounded by 
$$
	\frac{2}{\veps}\int_{X} |u_j-u| \om_{u_j}^2 + \frac{2C_n}{\veps} \max\{1,B\} \int_X |u_j- u| \om^2,
$$
which tends to zero as $j$ goes to infinity.
Since $\veps>0$ is fixed we conclude the convergence in capacity in dimension 2.
\end{proof}

\end{document}